\documentclass[10.5pt,letterpaper,leqno]{article}
\usepackage[letterpaper, hmargin=2cm, vmargin={2cm, 2cm}]{geometry}
\usepackage{myDefs}
\makeatletter
\newcommand{\oset}[3][0ex]{%
  \mathrel{\mathop{#3}\limits^{
    \vbox to#1{\kern-2\ex@
    \hbox{$\scriptstyle#2$}\vss}}}}
\makeatother

\newcommand{\old}[1]{\textcolor{black}{#1}}

\begin{document}
\title{An approximate solution to Erd\Humlaut{o}s'  maximum modulus points~problem}
\author
{Adi Gl\"ucksam {\small and} Leticia Pardo-Sim\'{o}n }

\date{}

\maketitle

\begin{abstract}
In this note we investigate the asymptotic behavior of the number of maximum modulus points, of an entire function, sitting in a disc of radius $r$. In 1964, Erd\Humlaut{o}s asked whether there exists a non-monomial function so that this quantity is unbounded? tends to infinity? In 1968 Herzog and Piranian constructed an entire map for which it is unbounded. Nevertheless, it is still unknown today whether it is possible that it tends to infinity or not. In this paper, we construct a transcendental entire function that is arbitrarily close to satisfying this property, thereby giving the strongest evidence supporting a positive answer to this question.
\end{abstract}
\let\thefootnote\relax
\vspace{-3em}

\section{Introduction}
Let $f$ be an entire function, and recall the \emph{maximum modulus function} defined for every $r\ge0$ by
$$
M(r)=M_f(r) := \underset{{|z| = r}}\max \abs{f(z)}.
$$
The \textit{maximum modulus set} is the set of points where $f$ achieves its maximum modulus, i.e.,
$$
 \{ |f(z)| = M(|z|) \}.
$$
Every point in the maximum modulus set is called a \textit{maximum modulus point}. If $f$ is a monomial, then the maximum modulus set is the entire plane. Otherwise, the maximum modulus set consists of a union of closed curves, which are analytic except at their endpoints; see \cite[Theorem~10]{valironlectures} or \cite{Blumenthal}, and might even be a single point; \cite{tyler,letidavenew}. Following uniqueness of entire functions, unless $f$ is a monomial, for each $r>0$, the intersection of the maximum modulus set with $\bset{\abs z=r}$  contains finitely many maximum modulus points of modulus $r$; see e.g. \cite[II.3]{valironlectures}. Let us denote by $v(r)$ the number of such points; that is,
\begin{equation*}
v_f(r)=v(r):= \# (\{ \vert z \vert=r \text{ and } \abs{f(z)}=M(\abs z)\}).
\end{equation*}
In 1964, P. Erd\Humlaut{o}s asked about the limiting behaviour of $v(r)$ as $r$ goes to infinity. More precisely, he posed the following questions:
\begin{quest}[Erd\Humlaut{o}s]
Let $f$ be a non-monomial entire function.
	\begin{enumerate}[label=(\alph*)]
		\item \label{item_Erdos1} Can the function $v$ be unbounded, i.e., $\limitsup r\infty v(r)=\infty$?
		\item \label{item_Erdos2} Can the function $v$ tend to $\infty$, i.e., $\limitinf r\infty v(r)=\infty$?
	\end{enumerate}
\end{quest}

The first time this problem seems to have appeared in the literature was in Hayman's book from 1967 on open problems in function theory \cite{Hayman_book_original}; also in its 50th anniversary edition as \cite[Problem 2.16]{Hayman_book_50}. The second part of the question was later also attributed to Clunie, appearing in Anderson, Barth and Brannan's compilation of open problems in complex analysis from 1977, \cite[Problem 2.49]{Anderson_problems}, as well as in \cite[Problem 2.49]{Hayman_book_50}.

In 1968, Herzog and Piranian answered  \ref{item_Erdos1} positively by constructing an entire function $f$ for which $v(n)=n$ for each $n\in \N$; \cite{piranian}. Their technique relies on a sequence of refinements that do not seem to provide any information nor control on the value of $v(r)$ for $r\notin \N$, and so, in their case, nothing can be deduced for \ref{item_Erdos2}. 

In fact, \ref{item_Erdos2} has eluded an answer to this date. Hayman showed in \cite{Hayman_origin} that, near the origin,  for any entire $f$, the maximum modulus set consists of a collection of disjoint analytic curves and provided an upper bound for the number of these curves in terms of their Taylor expansion; see \cite{mio_Vasso_Dave} for a refinement of Hayman's result. If $f$ is a polynomial, then, by considering its reciprocal, we see that analogous results hold for the structure of the maximum modulus set near infinity and, in particular, $v(r)$ is constant for all $r$ large enough; see \cite[Prop. 3.3]{letidavenew} for details. However, to the best of our knowledge, Herzog and Piranian's is the unique example in the literature of a (transcendental) entire function for which $v(r)$ is unbounded, and, a priori, it is not clear what to expect for the limiting behaviour of $v(r)$ in the transcendental case.

In this paper, we construct a transcendental entire function $f$ that is arbitrarily close to satisfying \ref{item_Erdos2}. More precisely, for each $r,\eps>0$, let us denote by $v(r,\eps)$ the number of connected components of the intersection of $\bset{\abs z=r}$ with the {\it$\eps$-approximately maximum modulus set} , i.e., 
\begin{equation}\label{eq_set_intervals}
\{\vert f(z)\vert>M(\abs z) -\eps\}. 
\end{equation}
Similarly, for every $r,\eps>0$ we denote by $w(r,\eps)$ the number of connected components of the intersection of $\bset{\abs z=r}$ with the {\it$\eps$-approximately zero set}, i.e.,  $\{\vert f(z)\vert<\eps\}$.
\begin{thm}\label{thm:max_min_mod}
	There exists an entire function $f$ satisfying that for every $\eps>0$
$$
	\limit r\infty v(r,\eps)=\infty \;\; \text{ and }\;\;\limit r\infty w(r,\eps)=\infty.
$$
\end{thm}

In fact, Theorem \ref{thm:max_min_mod} shows that there exists an entire $f$ such that for every $\eps>0$ and $N\in \N$, for all $r$ sufficiently large, the circle of radius $r$ contains at least $N$ arcs where $\vert f \vert$ attains values $\eps$-close to the maximum modulus separated by at least $N$ arcs where $\abs f$ is bounded by $\eps$. 

We prove Theorem \ref{thm:max_min_mod} by using an approximation technique based on H\"ormander's solution to $\overline\partial$-equations (see Theorem \ref{thm:Hormander} in Section 2). Similar approximation techniques were presented by Russakovskii in \cite{Russakovskii1993} and subsequently used by the first author in \cite{BGLS2017}, and \cite{Glucksam2017}, and with the second in \cite{Wandering2023}. 

Because our methods rely on approximation theory, it is theoretically possible that the number of intervals composing the intersection of $\bset{\abs z=r}$ with the $\eps$-approximately maximum modulus set, defined in (\ref{eq_set_intervals}), that actually contain a maximum modulus point, is uniformly bounded. We are therefore unable to conclude that \ref{item_Erdos2} holds for $f$. Still, we believe that our result strongly suggests that Erd\Humlaut{o}s' question \ref{item_Erdos2} should have a positive answer.

\begin{rmk} The function we construct in Theorem \ref{thm:max_min_mod} is of infinite lower order of growth, i.e., 
\begin{equation*}
\liminf_{r\to \infty}\frac{\log \log M(r)}{\log r}=\infty,
\end{equation*}
see estimates in the proof of Lemma \ref{lem_final}. In  \cite{Marchenko_conjecture_erdos}, Marchenko conjectured that for entire functions of finite lower order, the answer to \ref{item_Erdos2} should be negative. This supports the conjecture that any entire function  satisfying the property of Theorem \ref{thm:max_min_mod} should have infinite order, as in the example above.
\end{rmk}

\subsection*{Acknowledgments} This project originated from discussions that followed a talk given by the second author in the \textit{Complex Analysis online seminar series ``CAvid''}. We thank Rod Halburd for organizing it. We are also grateful to Jack Burkart and David Sixsmith for many helpful discussions, and to the referee, whose thoughtful and meticulous report has greatly improved this paper. This material is based upon work supported by the National Science Foundation under Grant No. 1440140, while the authors were in residence at the Mathematical Sciences Research Institute in Berkeley, California, during the winter semester of 2022.

\section{Definitions, notation, and preliminary results}
\subsection{The building blocks}

For every $n\in \N$, define the function $e_n\colon \C\to \C$ as 
\begin{equation*}
e_n(z):=\exp\bb{z^n}.
\end{equation*}
Note that this function satisfies that for every $r>0$, $v_{e_n}(r)=n$, while between every two maximum modulus points, there is an arc where $\abs{e_n}<1$.

We begin by providing an overview of our strategy to prove Theorem  \ref{thm:max_min_mod} (precise definitions will follow). We shall construct an entire function $f$ with the following property: There is a strictly increasing sequence of radii, $\{\rho_n\}$, and a sequence of sectors, $\bset{S_n}$, such that for all $n$ large enough, on the intersection of the sector $S_n$ with the annulus $\{z\colon \rho_n\leq \vert z \vert \leq \rho_{n+1}\}$, $f$ behaves approximately like a constant, $a_n$, times $e_{2^n}(z)$, while outside this set $\vert f\vert$  is strictly smaller than the maximum of $\vert f\vert$ inside the intersection. Combining this property with the observation we made on $\vert e_{2^n}\vert $ and carefully controlling the errors, we will show that Theorem \ref{thm:max_min_mod} holds for $f$.

\begin{figure}[htp]
	\centering
\begingroup%
\makeatletter%
\providecommand\color[2][]{%
	\errmessage{(Inkscape) Color is used for the text in Inkscape, but the package 'color.sty' is not loaded}%
	\renewcommand\color[2][]{}%
}%
\providecommand\transparent[1]{%
	\errmessage{(Inkscape) Transparency is used (non-zero) for the text in Inkscape, but the package 'transparent.sty' is not loaded}%
	\renewcommand\transparent[1]{}%
}%
\providecommand\rotatebox[2]{#2}%
\newcommand*\fsize{\dimexpr\f@size pt\relax}%
\newcommand*\lineheight[1]{\fontsize{\fsize}{#1\fsize}\selectfont}%
\ifx\svgwidth\undefined%
\setlength{\unitlength}{266.45669291bp}%
\ifx\svgscale\undefined%
\relax%
\else%
\setlength{\unitlength}{\unitlength * \real{\svgscale}}%
\fi%
\else%
\setlength{\unitlength}{\svgwidth}%
\fi%
\global\let\svgwidth\undefined%
\global\let\svgscale\undefined%
\makeatother%
\begin{picture}(1,0.59575531)%
	\lineheight{1}%
	\setlength\tabcolsep{0pt}%
	\put(0,0){\includegraphics[width=\unitlength,page=1]{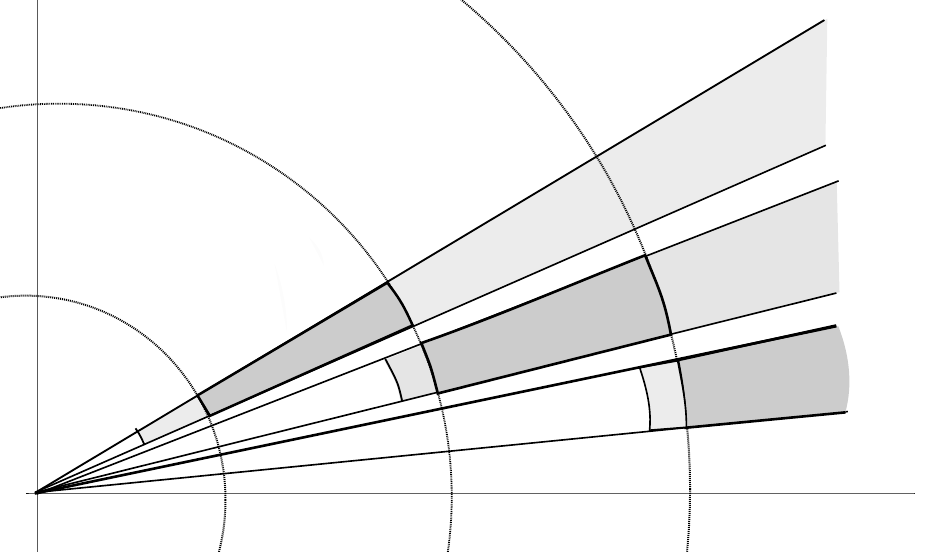}}%
	\put(0.72207902,0.42449194){\color[rgb]{0,0,0}\makebox(0,0)[lt]{\lineheight{1.25}\smash{\begin{tabular}[t]{l}$\fontsize{9pt}{1em}S_{n-1}$\end{tabular}}}}%
	\put(0.79989477,0.17501065){\color[rgb]{0,0,0}\makebox(0,0)[lt]{\lineheight{1.25}\smash{\begin{tabular}[t]{l}$\fontsize{9pt}{1em}f\approx a_{n+1}e_{2^{n+1}}$\end{tabular}}}}%
	\put(0.78412507,0.30889424){\color[rgb]{0,0,0}\makebox(0,0)[lt]{\lineheight{1.25}\smash{\begin{tabular}[t]{l}$\fontsize{9pt}{1em}S_{n}$\end{tabular}}}}%
	\put(0.49245807,0.03012839){\color[rgb]{0,0,0}\makebox(0,0)[lt]{\lineheight{1.25}\smash{\begin{tabular}[t]{l}$\fontsize{9pt}{1em}\rho_{n}$\end{tabular}}}}%
	\put(0.25081867,0.03008849){\color[rgb]{0,0,0}\makebox(0,0)[lt]{\lineheight{1.25}\smash{\begin{tabular}[t]{l}$\fontsize{9pt}{1em}\rho_{n-1}$\end{tabular}}}}%
	\put(0.74918191,0.02728446){\color[rgb]{0,0,0}\makebox(0,0)[lt]{\lineheight{1.25}\smash{\begin{tabular}[t]{l}$\fontsize{9pt}{1em}\rho_{n+1}$\end{tabular}}}}%
	\put(0.51020954,0.22281242){\color[rgb]{0,0,0}\makebox(0,0)[lt]{\lineheight{1.25}\smash{\begin{tabular}[t]{l}$\fontsize{9pt}{1em}f\approx a_ne_{2^n}$\end{tabular}}}}%
\end{picture}%
\endgroup%
	\caption{Schematic of the action of $f$.  The (infinite) sectors $(S_n)$ are depicted in light gray. Each of them contains a dark-gray region, where $f$ acts approximately like the map $a_n\cdot e_{2^n}$, and so the maximum modulus properties of $a_n\cdot e_{2^n}$ are transferred to $f$, up to an error.}
	\label{fig:idea}
\end{figure}

We will obtain the function $f$ as the limit of a sequence of entire functions, $\{f_n\}$, defined inductively. Roughly speaking, $f_n$ will act, up to some error, as follows: It will agree with $f_{n-1}$ in a region $P_n$, where $P$ stands for ``past'', will be defined as $a_n\cdot e_{2^{n}}$ in the two sectors $S_n$, and will be equal to $0$ in a sector $F_n$ around the real axis, where $F$ stands for ``future''. The region $P_n$ will contain the ball of radius $r_n\approx \rho_n+3$, and this property will be used to show uniform convergence of the sequence $\{f_n\}$. The regions $P_n$, $S_n$, and $F_n$ will be disjoint and their union will be \textit{ almost} equal to the whole complex plane; see Figure~\ref{fig:pasting_zones} below. An important feature of our construction is that $$(S_m\cup F_m)\subset F_{n-1} \quad \text{ for all }\quad m\geq n.$$
In a sense, we are pasting three different holomorphic functions, $f_{n-1}(z), a_ne_{2^n}(z)$ , and $0$, in three disjoint regions, $ P_n, S_n ,$ and $F_n$. It is not surprising that our techniques require some `glueing areas' to allow smoothness. These `glueing areas' are the sets $B_n$ and $T_n$, defined below, which naturally must contain $\mathbb C\setminus (P_n\cup S_n\cup F_n)$.

For simplicity, our sets will be symmetric with respect to the $x$-axis.
	
\begin{figure}[htp]
	\centering
\begingroup%
\makeatletter%
\providecommand\color[2][]{%
	\errmessage{(Inkscape) Color is used for the text in Inkscape, but the package 'color.sty' is not loaded}%
	\renewcommand\color[2][]{}%
}%
\providecommand\transparent[1]{%
	\errmessage{(Inkscape) Transparency is used (non-zero) for the text in Inkscape, but the package 'transparent.sty' is not loaded}%
	\renewcommand\transparent[1]{}%
}%
\providecommand\rotatebox[2]{#2}%
\newcommand*\fsize{\dimexpr\f@size pt\relax}%
\newcommand*\lineheight[1]{\fontsize{\fsize}{#1\fsize}\selectfont}%
\ifx\svgwidth\undefined%
\setlength{\unitlength}{334.48818898bp}%
\ifx\svgscale\undefined%
\relax%
\else%
\setlength{\unitlength}{\unitlength * \real{\svgscale}}%
\fi%
\else%
\setlength{\unitlength}{\svgwidth}%
\fi%
\global\let\svgwidth\undefined%
\global\let\svgscale\undefined%
\makeatother%
\begin{picture}(1,1.4406865)%
	\lineheight{1}%
	\setlength\tabcolsep{0pt}%
	\put(0,0){\includegraphics[width=\unitlength,page=1]{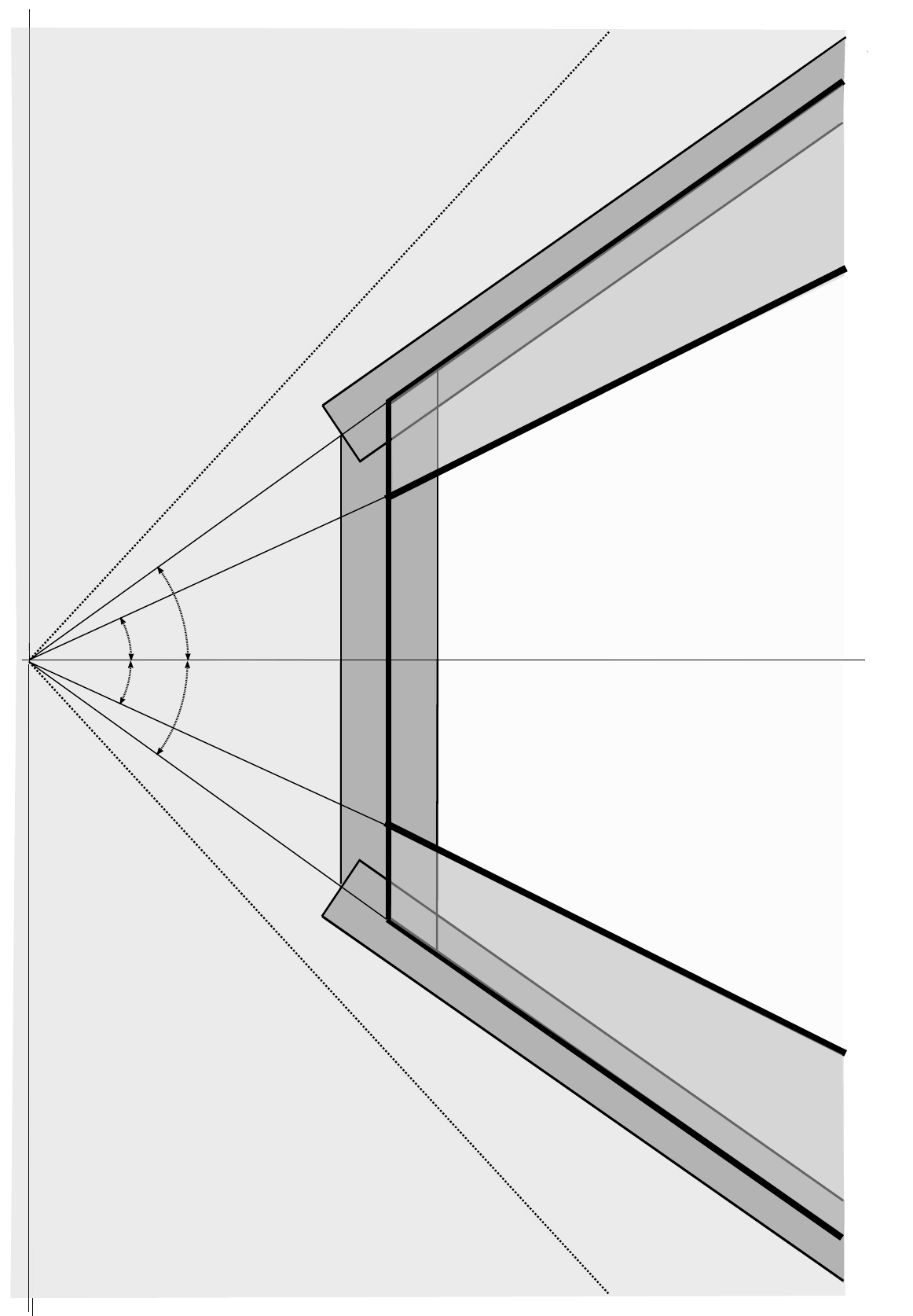}}%
	\put(0.92819124,0.03643731){\color[rgb]{0,0,0}\makebox(0,0)[lt]{\lineheight{1.25}\smash{\begin{tabular}[t]{l}$\fontsize{9pt}{1em}T_n$\end{tabular}}}}%
	\put(0.39259868,0.68164096){\color[rgb]{0,0,0}\makebox(0,0)[lt]{\lineheight{1.25}\smash{\begin{tabular}[t]{l}$\fontsize{8pt}{1em}r_n\!+\!\frac12$\end{tabular}}}}%
	\put(0.43515227,0.8296519){\color[rgb]{0,0,0}\makebox(0,0)[lt]{\lineheight{1.25}\smash{\begin{tabular}[t]{l}$\fontsize{9pt}{1em}\beta_n$\end{tabular}}}}%
	\put(0.93300686,1.14075523){\color[rgb]{0,0,0}\makebox(0,0)[lt]{\lineheight{1.25}\smash{\begin{tabular}[t]{l}$\fontsize{9pt}{1em}\gamma^+_n$\end{tabular}}}}%
	\put(0.93337684,1.33785807){\color[rgb]{0,0,0}\makebox(0,0)[lt]{\lineheight{1.25}\smash{\begin{tabular}[t]{l}$\fontsize{9pt}{1em}\ell_n^+$\end{tabular}}}}%
	\put(0.93364205,1.38514502){\color[rgb]{0,0,0}\makebox(0,0)[lt]{\lineheight{1.25}\smash{\begin{tabular}[t]{l}$\fontsize{9pt}{1em}T_n$\end{tabular}}}}%
	\put(0.43393809,0.73549277){\color[rgb]{0,0,0}\makebox(0,0)[lt]{\lineheight{1.25}\smash{\begin{tabular}[t]{l}$\fontsize{9pt}{1em}B_n$\end{tabular}}}}%
	\put(0.21321566,1.10344024){\color[rgb]{0,0,0}\makebox(0,0)[lt]{\lineheight{1.25}\smash{\begin{tabular}[t]{l}$\fontsize{9pt}{1em}P_n$\end{tabular}}}}%
	\put(0.72151642,1.10501154){\color[rgb]{0,0,0}\makebox(0,0)[lt]{\lineheight{1.25}\smash{\begin{tabular}[t]{l}$\fontsize{9pt}{1em}S_n$\end{tabular}}}}%
	\put(0.75033208,0.28189278){\color[rgb]{0,0,0}\makebox(0,0)[lt]{\lineheight{1.25}\smash{\begin{tabular}[t]{l}$\fontsize{9pt}{1em}S_n$\end{tabular}}}}%
	\put(0.71991526,0.78614663){\color[rgb]{0,0,0}\makebox(0,0)[lt]{\lineheight{1.25}\smash{\begin{tabular}[t]{l}$\fontsize{9pt}{1em}F_n$\end{tabular}}}}%
	\put(0.55377839,1.38091308){\color[rgb]{0,0,0}\rotatebox{0.13015559}{\makebox(0,0)[lt]{\lineheight{1.25}\smash{\begin{tabular}[t]{l}$\fontsize{9pt}{1em}F_{n-1}$\end{tabular}}}}}%
	\put(0.20692758,0.75473602){\color[rgb]{0,0,0}\makebox(0,0)[lt]{\lineheight{1.25}\smash{\begin{tabular}[t]{l}$\fontsize{7pt}{1em}t_n+\Theta_n$\end{tabular}}}}%
	\put(0.14973445,0.73434909){\color[rgb]{0,0,0}\makebox(0,0)[lt]{\lineheight{1.25}\smash{\begin{tabular}[t]{l}$\fontsize{7pt}{1em}t_n$\end{tabular}}}}%
	\put(0.92640912,0.27788836){\color[rgb]{0,0,0}\makebox(0,0)[lt]{\lineheight{1.25}\smash{\begin{tabular}[t]{l}$\fontsize{9pt}{1em}\gamma^-_n$\end{tabular}}}}%
	\put(0.93137766,0.06852884){\color[rgb]{0,0,0}\makebox(0,0)[lt]{\lineheight{1.25}\smash{\begin{tabular}[t]{l}$\fontsize{9pt}{1em}\ell_n^-$\end{tabular}}}}%
\end{picture}%
\endgroup%
	\caption{In different shades of gray, some of the sets and parameters involved in the construction of the function $f_n$. In bold, the lines $\beta_n$, $\ell_n^{\pm}$ and $\gamma^{\pm}_n$.}
	\label{fig:pasting_zones}
\end{figure}

We will frequently use the following notation: Given a set $U\subset\C$ and $\delta>0$ we let
$$
U^{-\delta}:=\bset{z \colon \dist(z, \C\setminus U)>\delta},
$$
i.e., the set $U^{-\delta}$ is the set of all points in $U$ which are separated from the boundary of $U$ by at least $\delta$. Similarly, we define
$$
U^{+\delta}:=\bset{z \colon \dist(z, U)<\delta},
$$
i.e., the set $U^{+\delta}$ is a `fattening' or `padding' of the set $U$ by $\delta$.

We now provide precise definitions of the sets and parameters used in the construction, with the definitions above in mind. For every $n\ge 13$,
\begin{itemize}
	\item Parameters:
		\begin{align*}
			&\Theta_n:=2\pi\cdot2^{-n}\cdot\bb{n+\frac{1}6}&&t_n:=\sumit m {n+1}\infty2\Theta_m=4\pi\cdot 2^{-n}\bb{n+2+\frac16}\\
			&r_n:=\frac{128n}{t_n}+1 && a_n:=\exp\bb{-\bb{r_n+3}^{2^n}}.
		\end{align*}
	
	\item Main regions: 
	\begin{align*}
			P_n&:=\bset{z=re^{it} \colon \Re(z)<r_n \} \cup \{z=re^{it} \colon t_n+\Theta_n<\abs t<\pi},\\
		S_n& :=\bset{z=re^{it} \colon \Re(z)\ge r_n+\frac12\text{ and } \abs t\in\sbb{t_n,t_n+\Theta_n}},\text {and }\\
	F_n &:=\bset{z=re^{it} \colon \Re(z)>r_n +\frac12 \text{ and } \abs t<t_{n}}.
	\end{align*}
	\item Auxiliary regions:
	
		Denote by $\Sigma$ the half-strip
		$\Sigma:=\{z\colon \Re(z)>0 \text{ and } \vert \Im(z)\vert\leq 1/2 \}.$
	Then, define the sets
	\begin{align*}
		T_n &:= \{z \colon e^{-i\bb{t_n+\Theta_n}}z-r_n\in \Sigma\} \cup \{z \colon e^{i\bb{t_n+\Theta_n}}z-r_n\in \Sigma\}.\\
		B_n &:= \bset{z=re^{it} \colon \Re(z)\in [r_n, r_{n}+1] \text{ and } \abs{t} \leq t_n+\Theta_n}.
	\end{align*}
Note that $T_n$ is the union of two translated and rotated copies of $\Sigma$.

\item Auxiliary	lines:
	\begin{align*}
		\ell_n^+&:=\bset{z=re^{it} \colon \Re(z)\ge r_n + \frac12 \text{ and }  t= t_n+\Theta_n},\\
		\ell_n^-&:=\bset{z=re^{it} \colon \Re(z)\ge r_n + \frac12 \text{ and }  -t= t_n+\Theta_n},\\
		\gamma_n^+&:=\bset{z=re^{it} \colon \Re(z)\ge r_n + \frac12 \text{ and }  t= t_n},\\
		\gamma_n^-&:=\bset{z=re^{it} \colon \Re(z)\ge r_n + \frac12 \text{ and }  -t= t_n}, \text{ and }\\
		\beta_n &:=\bset{z=re^{it} \colon \Re(z)=r_n+\frac12 \text{ and } \abs{t} \leq t_n+\Theta_n}.
	\end{align*}
		
	Observe that  $\ell^{\pm}_n$ are essentially the central lines of the two components of $T_n$, and $\beta_n$ is the central line of $B_n$. Moreover, $(\ell^{\pm}_n)^{+\frac14} \subset T_n$ and $\beta_n^{+\frac14} \subset (B_n\cup T_n)$, while $\gamma_n^\pm$ are in $\partial S_n\cap \partial F_n$;  see Figure~\ref{fig:pasting_zones}.
\end{itemize}

The following observation describes several estimates that will be use throughout the proof of our theorem.

\begin{obs}\label{obs_N0}
For all $n\in\N$ large enough
\begin{enumerate}[label=(\alph*)]
\item $r_n\in (10\cdot 2^n, 11\cdot 2^n)$.
\item $r_{n}^{-2^{n-1}}+e^{-\frac1{12}}<1.$
\item $\frac{\exp\bb{\pi r_n}}{10^3}>e^{20n}$.
\item $3(1+r_{n+2})^2\sum_{j\geq n-1}2^{-4j}\leq 2^{-n}$.
\end{enumerate}
\end{obs}
\begin{proof}
Recall that $r_n=\frac{128\cdot n}{t_n}+1=\frac{2^{n}\cdot 128\cdot n}{4\pi (n+13/6)}+1= 2^n\bb{\frac{128}{4\pi}\cdot \frac{1}{1+13/(6n)}+\frac{1}{2^n}}$. A calculation shows that (a) in particular holds for all $n\geq 200.$ Following (a), properties (b) and (c) clearly hold for all $n$ large enough, as does (d), since  $$3(1+r_{n+2})^2\sum_{j\geq n-1}2^{-4j}\leq 2r_{n+2}^2\cdot 2^{-4n+5}\leq \frac{2(11\cdot2^{n+2})^2}{2^{4n-5}}\leq \frac{1}{2^n}$$
for sufficiently large $n$.
\end{proof}


\subsection{The pasting process and preliminaries}
To be able to define an entire function, that approximates different functions in disjoint regions, we will use H\"ormander's theorem on the solution to the $\bar\partial$-equation; see \cite[Theorem 4.2.1]{Hormander}.
\begin{thm}[H\"ormander]\label{thm:Hormander}
Let $u:\C\rightarrow\R$ be a subharmonic function. Then, for every locally square integrable function $g\colon \C\to\C$ there is a locally square integrable solution $\alpha \old{\colon \C\to\C}$ of the equation $\bar\partial \alpha=g$ such that
\begin{equation*}
\int_\C\abs {\alpha(z)}^2\frac{e^{-u(z)}}{\bb{1+\abs z^2}^{2}}dm(z)\le\frac12 \int_\C \abs {g(z)}^2e^{-u(z)}dm(z),
\end{equation*}
provided that the integral on the right hand side is finite.
\end{thm}
\begin{obs}\label{obs_Hormander} The equation $\bar\partial \alpha=g$ in the statement holds in distribution. The application of the theorem in our construction will provide us with a function $f\colon \C\to\C$ such that $\bar\partial f=0$ in distribution. This is enough to conclude that $f$ is entire: note that any locally square integrable function $h$ so that $\Delta h=0$ in distribution is harmonic, since it must be both subharmonic and superharmonic in distribution and satisfy a mean value property. Taking $h(z)=\Re(f(z))$ and noting that $f$ is defined in a simply-connected domain, our claim follows, as then $h$ has a conjugate, $\tilde{h}$, such that $h+\tilde{h}i$ is entire and differs from $f$ by at most a constant.
\end{obs}

Reading the statement, one can see we will need an underlying subharmonic function. This function guarantees the integral on the right indeed converges, and later we will also see it controls the growth of the resulting function. For now, let us take care of the first component- the subharmonic function: we will construct a sequence of subharmonic functions $\bset{u_n}$ each one to be used in the construction of the function $f_n$ from its predecessor.

The function $g_n$ will be supported on a small set contained in $T_n\cup B_n$, where $u_n$ grows rapidly, making the integral, and therefore the error, very small. Outside $T_n\cup B_n$, $u_n$ will be zero, providing a good approximation of $f_n$.
\begin{lem}\label{lem:sh}
For every $n\in\N$ large enough, there exists a subharmonic function $u_n\colon \C \to [0, \infty)$ satisfying that
\begin{enumerate}
\item Upper Bound: $u_n(z)\le\exp\bb{2\pi\abs{z}}$.
\item $Spt(u_n)\subset B_n\cup T_n$.
\item Lower Bounds:
$$
u_n(z)\ge	\begin{cases}
			e^{20n},& z\in \beta_n^{+\frac14},\\
			\frac{1}{10^3}\exp\bb{\pi\abs z},& z\in (\ell^{\pm}_n)^{+\frac14}.
		\end{cases}
$$
\end{enumerate}
\end{lem}
\begin{proof}
We begin by defining
$$
v(z):=v(x+iy)=\begin{cases}
\max\bset{\exp\bb{\pi x}\cos\bb{\pi y}-1,0},& x>0\text{ and }\abs y<\frac12,\\
0,& \text{ otherwise}.
\end{cases}
$$
The function $v$ is subharmonic and satisfies that
\begin{enumerate}[label=(\alph*)]
\item $v(z)=0$ whenever $\Re(z)\le0$ or $\abs{\Im(z)}\ge\frac12$.
\item $v(z)\le\exp\bb{2\pi\abs z}$.
\item For $z$ with $\abs{\Im(z)}\in\bb{0,\frac13}$, $v(z)\ge\frac12\exp\bb{\pi \Re(z)}-1$.
\end{enumerate}
For every $n$ large enough so that Observation \ref{obs_N0} holds,  
we let
$$
v_n(z)=v\bb{e^{-i\bb{t_n+\Theta_n}}z-r_n}+v\bb{e^{i\bb{t_n+\Theta_n}}z -r_n},
$$
i.e., $v_n$ is the sum of two rotations and translations of the function $v$ moving the two components of $T_n$ to the half strip $\Sigma$; see Figure~\ref{fig:pasting_zones}. In particular, $v_n$ is supported inside $T_n$. Using property (c) of the function $v$, if $z\in (\ell_n^{\pm})^{+\frac14},$ then we have that  $\abs{\Im\bb{e^{\pm i\bb{t_n+\Theta_n}}z-r_n}}<\frac14$, and therefore
$$
v_n(z)\ge \frac12\exp\bb{\pi \bb{\Re\bb{e^{\pm i\bb{t_n+\Theta_n}}z-r_n}}}-1.
$$
Note that we can  write any $z\in (\ell_n^{\pm})^{+\frac14}$ as 
\begin{equation}\label{eq_zinl}
z=r\cdot e^{\mp i\bb{t_n+\Theta_n}}+\zeta \text{   for some   } \zeta\in \frac14\D   \text{   and   } r>r_n,
\end{equation}
and so
\begin{align*}
\Re\bb{e^{\pm i\bb{t_n+\Theta_n}}z-r_n}&=\Re\bb{e^{\pm i\bb{t_n+\Theta_n}}\bb{r\cdot e^{\mp i\bb{t_n+\Theta_n}}+\zeta}-r_n}\\
&=r-r_n+\Re\bb{e^{\pm i\bb{t_n+\Theta_n}}\zeta}\ge r-r_n-\frac14.
\end{align*}
This implies that if $z\in(\ell_n^{\pm})^{+\frac14}$, then
$$
v_n(z)\ge \frac12\exp\bb{\pi\bb{r-r_n -\frac14}}-1.
$$

We next define another subharmonic function 
$$
b_n(z)=b_n(x+iy)=	\begin{cases}
		\cosh(\pi y)\cos\bb{\pi\bb{x-\bb{r_n+\frac12}}},& x\in\bb{r_n,r_n+1},\\
		0,& \text{otherwise},
		\end{cases}
$$
and use it to define
$$
u_n(z)=	\begin{cases}
			\max\bset{2e^{20n}\cdot b_n(z), \exp(\pi r_n)\cdot v_n\bb{z}},& z\in B_n,\\
			 \exp(\pi r_n)\cdot v_n\bb{z},& \text{otherwise}.
		\end{cases}
$$
Away from $\partial B_n$ the function $u_n$ is subharmonic as maximum of such functions. It is left to verify that in the transition area, $\partial B_n$, the function equals $\exp(\pi r_n)\cdot v_n\bb{z}$, and therefore $u_n$ is well-defined and subharmonic. On the sides of $\partial B_n$ the function $b_n$ is $0$ and so there is nothing to check. On the top and bottom of $\partial B_n$, 
\begin{align*}
2e^{20n}\cdot b_n(z)&\le 2e^{20n}\cdot \cosh\bb{\pi\bb{r_n\cdot \bb{t_n+\Theta_n}+1}}\le 2e^{20n+\pi}\cdot \exp\bb{2\pi\cdot t_n\bb{\frac{128n}{t_n}+1}}\le e^{900n}\\
&\overset{(\star)}\le \exp(\pi r_n)\cdot \bb{\frac12\exp\bb{\pi\bb{r-r_n-\frac14}}-1}\le  \exp(\pi r_n)\cdot v_n\bb{z}
\end{align*}
where $(\star)$ holds for every $r>r_n$ following Observation \ref{obs_N0}(a), since $r_n\ge 2^{n}$. We conclude that $u_n$ is well defined and subharmonic as local maximum of subharmonic functions.

We will now show that properties 1, 2, and 3 hold for the function $u_n$:
\begin{enumerate}
\item Following property (b) of the function $v$, if $z\in B_n$,
$$
2e^{20n}\cdot b_n(z)\le 2e^{20n}\cdot \exp\bb{\pi \Im(z)}\le e^{900n}\le e^{2\pi r_n}\le e^{2\pi\abs z},
$$
following Observation \ref{obs_N0}(a),
while for every $z\in \C$,
$$
\exp(\pi r_n)\cdot v_n\bb{z}\le \exp\bb{\pi r_n+2 \pi\bb{\abs z-r_n+1}}\le e^{2\pi\abs z},
$$
concluding the proof of the first property.
\item Following property (a) of the function $v$, $u_n$ is supported inside $B_n\cup T_n$ and so the second property holds.
\item Lastly, to see the third property note that if $z\in (\ell_n^{\pm})^{+\frac14}$, then using its representation as in (\ref{eq_zinl}), by property~(c) of the function $v$, 
\begin{align*}
u_n(z)\ge & \exp(\pi r_n)\cdot v_n\bb{z}\ge \exp(\pi r_n)\bb{\frac12\exp\bb{\pi\bb{r-r_n-\frac14}}-1} \\
	=& \exp(\pi r)\bb{\frac{\exp\bb{-\frac\pi4}}{2}-\exp(\pi\bb{r_n-r})}\ge \exp(\pi r)\bb{\frac{\exp\bb{-\frac\pi4}}{2}-\exp\bb{-\frac{\pi}{2}}}\\
	&\geq \exp\bb{\pi\bb{r+\frac14}}\exp\bb{-\frac\pi4}\bb{\frac{\exp\bb{-\frac\pi4}}{2}-\exp\bb{-\frac{\pi}{2}}}\geq \frac{\exp\bb{\pi \abs{z}}}{10^3},
\end{align*}
as $r>r_n+\frac12$ and $\abs{z}<r+\frac14$.

Next, let $z\in \beta_n^{+\frac14}$. If $z\notin B_n$, then $z\in(\ell_n^{\pm})^{+\frac14}$, and we use the bound above and Observation \ref{obs_N0}(c). Otherwise, $z\in \beta_n^{+\frac14}\cap  B_n$, and
$$
u_n(z)\ge 2e^{20n}\cdot\frac12=e^{20n}
$$
as needed to conclude that property 3 holds.\qedhere
\end{enumerate}
\end{proof}
Given a sequence of uniformly separated domains $\Omega_1,\Omega_2,\cdots$ , we would like to find a smooth map whose zero set separates the domains from one another, assigns 1 to points in each one of the domains, and has a bounded gradient, which depends on the distance between the domains. The next proposition uses convolutions to construct such a map.
\begin{prop}\label{prop:chi}
There is a constant $C>0$ such that for every sequence of domains $\Omega_1,\Omega_2,\cdots$ satisfying that for every $j\ne k$,
$$
\dist(\Omega_j,\Omega_k):=\underset{z\in\Omega_k\atop w\in\Omega_j}\inf\; \abs{z-w}\ge 1,
$$
there exists a smooth mapping $\chi:\C\rightarrow[0,1]$ so that
\begin{enumerate}
\item For every $k$, $\chi|_{\Omega_k}\equiv 1$;
\item $\chi$ is supported on $ \bigcup_k\Omega_k^{+\frac14};$
\item $\nabla\chi$ is supported on $ \bigcup_k\bb{\Omega_k^{+\frac14}\setminus \Omega_k}$ and $\underset{z\in\C}\sup\;\abs{\nabla\chi(z)}\le C$.
\end{enumerate}
\end{prop}
\begin{proof}[Sketch of proof]
We only include an idea of the proof here as this is an exercise in convolutions. Recall the bump function
$$
b(z):=	\begin{cases}
			A\exp\bb{\frac{-1}{1-25\abs z^2}},& \abs z<\frac15,\\
			0,& \abs z\ge \frac15,
			\end{cases}
$$
where the constant $A$ is chosen such that $A\int b(z)dm(z)=1$. We define
$$
t(z):= \max\bset{0,4\bb{1-\min_k \dist\bb{z,\Omega_k^{+\frac15}}}-3}.
$$
The function $t\bb\cdot$ is continuous, well defined, has image $[0,1]$, and if $z\in \Omega_k$, then $t|_{B\bb{z,\frac15}}\equiv 1$, while if for every $k$, $\dist\bb{z,\Omega_k}>\frac14$ then $t(z)=0$. We then define
$$
\chi(z)=(b\ast t)(z).
$$
We leave it an exercise to verify that properties 1-3 hold.
\end{proof}
\begin{obs}
If in the previous proposition the sequence of domains $\{\Omega_i\}$ satisfies that $\dist(\Omega_j,\Omega_k)\geq 1/n$ for all $j\neq k$ and some $n\geq 1$, then a calculation shows that the constant in statement becomes $C\cdot n$ and, for each $k$, $\Omega_k^{+\frac14}$ becomes~$\Omega_k^{+\frac1{4n}}$.
\end{obs}
\section{A sequence of entire functions}
In this section we will construct the sequence of entire functions, $\bset{f_n}$. These functions satisfy that $f_n$ agrees (up to an error term) with $f_{n-1}$ on the past, $P_n$, (excluding some pasting area), and on the future, $F_n$, while on the sectors, $S_n$, the function $f_n$ is (up to an error term) $a_n\cdot \exp\bb{z^{2^n}}$.

Recall that for a set $\Omega\subset\C$ we defined $\Omega\inv:=\bset{z \colon \dist(z \colon \C\setminus\Omega)>1}$.
\begin{lem}\label{lem:sequence_liminf}
There exist $N_1$ large enough, so that Observation \ref{obs_N0} holds, and sequences of functions $\bset{h_n}_{n\geq N_1}$, $h_n\colon \C\to \C$, and $\bset{f_n}_{n\geq N_1}$, $f_n\colon \C\to \C$, so that $$f_{N_1}(z):=h_{N_1}(z):=a_{N_1}e_{2^{N_1}}(z)$$
 and for all $n\geq N_1$:
\begin{enumerate}
\item For $n>N_1$,
$$
h_n(z)=	\begin{cases}
		f_{n-1}(z),& z\in P_n\inv,\\
		a_ne_{2^n}(z),& z\in S_n\inv,\\
		0,& z\in F_n\inv.
		\end{cases}
$$
\item $\underset{z\in P_n\inv\cup S_n\inv\cup F_n^{-1}}\sup\frac{\abs{f_n(z)-h_n(z)}}{\bb{1+\abs z}^2}<2^{-4n}$.
\item For every $z\in\C\setminus\bb{P_n\inv\cup S_n\inv\cup F_n^{-1}}$ we have $ \abs{f_n(z)}\le 3 \bb{1+\abs z}^2\exp\bb{\exp\bb{7 \abs z}}$.
\item $f_n$ is entire.
\end{enumerate}
\end{lem}
The rest of this section will be dedicated to proving this lemma, and $N_1$ will be chosen at the end.

For every $n\geq N_1$, let $u_n$ be the subharmonic function defined in Lemma \ref{lem:sh}. Define the set $L:=\beta_n \cup \ell^{\pm}_n\cup \gamma^\pm_n$ which is connected (see Figure \ref{fig:pasting_zones}). The set $\C\setminus L^{+\frac14}$ consists of four connected components, that are at distance at least $1/2$ from one another. Moreover, the sets $P^{-\frac14}_n,\;S_n^{-\frac14}$ and $F_n^{-\frac14}$ lie in different connected components of $\C\setminus L^{+\frac14}$. Thus, we can apply Proposition \ref{prop:chi} to these four components to get a smooth function $\chi_n\colon \C\to [0,1]$ which, in particular, equals $1$ on each of the sets $P_n^{-\frac14},\;S_n^{-\frac14}$ and $F_n^{-\frac14}$, $spt\nabla\chi_n\subset L^{+\frac14}\setminus L^{+\frac18}$, $\chi_n\equiv 0$ in $L^{+\frac18}$ and $\underset{z\in\C}\sup\;\abs{\nabla\chi_n(z)}\le 2C'$, for some $C'>0$ independent of $n$. We denote $C:=2C'$.

Let $\Omega_1^n$ denote the connected component of $\C\setminus\bset{\chi_n\equiv 0}$ containing $P_n^{-\frac14}$, $\Omega_2^n$ be the union of the two connected components containing $S_n^{-\frac14}$, and $\Omega_3^n$ be the connected component containing $F_n^{-\frac14}$. 

We shall construct the functions $f_n$ recursively; Note that the properties in the statement of the lemma hold trivially for $f_{N_1}$ and $h_{N_1}$. Suppose now that $f_{n-1}$ has been defined and properties 1-4 hold for $f_{n-1}$ and $h_{n-1}$.

We define the model map
\begin{align*}
h_n(z)&:=	\chi_n(z)\cdot \bb{f_{n-1}(z)\cdot\indic{\Omega_1^n}(z)+a_n\cdot e_{2^n}(z)\cdot\indic{\Omega_2^n}(z)}
=		\begin{cases}
			\chi_n(z)\cdot f_{n-1}(z),&  z\in\Omega_1^n,\\
			\chi_n(z)\cdot a_n\cdot e_{2^n}(z),& z\in\Omega_2^n,\\
			0,& \text{otherwise}.
		\end{cases}
\end{align*}
The map $h_n(z)$ is well defined and smooth as a sum and product of smooth functions. Note that if $z\in P_n\inv\cup S_n\inv\cup F_n^{-1}$, then $\chi_n(z)=1$ and therefore property 1 of the sequences holds.

We would like to use H\"ormander's Theorem with the functions 
$$
g(z)=g_n(z):=\bar\partial h_n(z)=\begin{cases}
			\bar\partial\chi_n(z)\cdot f_{n-1}(z),&  z\in\Omega_1^n,\\
			\bar\partial\chi_n(z)\cdot a_n\cdot e_{2^n}(z),& z\in\Omega_2^n,\\
			0,& \text{otherwise},
		\end{cases}
$$
and the subharmonic function $u(z)=u_n(z)$ from Lemma \ref{lem:sh}. 
\subsection{Bounding the integral $\int_\C\abs {g_n(z)}^2e^{-u_n(z)}dm(z)$:}
In order to use H\"ormander's theorem, we first need to bound from above the integral
$$
\int_\C \abs {g_n(z)}^2e^{-u_n(z)}dm(z).
$$
Indeed, following property 4 of the function $\chi_n$, there exists some uniform constant $C>0$ so that
\begin{align*}
\int_\C \abs {g_n(z)}^2e^{-u_n(z)}dm(z)&=\int_{spt{\nabla\chi_n}}\abs {g_n(z)}^2e^{-u_n(z)}dm(z)\le C^2\bb{I_1+I_2}\\
\text{ for }\;\;\;\;\;\;\;\;\;\;I_1&:=\int_{spt{\nabla\chi_n}\cap\Omega_1^n}\abs {f_{n-1}(z)}^2e^{-u_n(z)}dm(z)\\
I_2&:=\int_{spt{\nabla\chi_n}\cap\Omega_2^n}a_n^2\abs {e_{2^n}(z)}^2e^{-u_n(z)}dm(z).
\end{align*}
Let us bound each integral separately. Note that $spt\nabla\chi_n\cap\Omega_1^n\subset \beta_n^{+\frac14}\cup (\ell^{\pm}_n)^{+\frac14}\subset F_{n-1}$ and, since we assumed that $f_{n-1}$ satisfies property 2,
\begin{align*}
I_1&=\int_{spt{\nabla\chi_n}\cap\Omega_1^n}\abs {f_{n-1}(z)}^2e^{-u_n(z)}dm(z)\le \int_{spt{\nabla\chi_n}\cap\Omega_1^n}2^{-8(n-1)}\bb{1+\abs z}^4e^{-u_n(z)}dm(z)\\
& \leq 2^{-8(n-1)}\bb{\int_{(\ell^{\pm}_n)^{+\frac14}\cap\Omega_1^n}\bb{1+\abs z}^4e^{-u_n(z)}dm(z)+\int_{\beta_n^{+\frac14}\cap\Omega_1^n}\bb{1+\abs z}^4e^{-u_n(z)}dm(z)}\\
&\le2^{-8(n-1)}\bb{2\integrate{r_n+\frac14}{\infty}{\bb{1+ r}^4\exp\bb{-\frac{e^{\pi r}}{10^3}}}r+2\exp\bb{-e^{20n}}\bb{\bb{r_n+1}^2+\bb{r_n\cdot \bb{t_n+\Theta_n}+1}^2}^2m(\beta_n^{+\frac14})}\\
&\le2^{-8(n-1)}\bb{4\integrate{2^{n+1}}{\infty}{r^4\exp\bb{-r^5}}r+1100n\exp\bb{-e^{20n}}\bb{(550n)^2+2^{2n+6}}^2}\le\frac{2^{-10n}}{2C^2}
\end{align*}
since $r^5<\frac{\exp(\pi r)}{10^3}$ for all $r\in\R$ large enough, and we may assume $n\geq N_1$ with $N_1$ large enough. 

To bound $I_2$, we note that, since $spt{\nabla\chi_n}\cap\Omega_2^n\subset \beta_n^{+\frac14}\cup (\ell^{\pm}_n)^{+\frac14}\cup (\gamma^\pm_n)^{+\frac14}$, then
\begin{align*}
I_2&=\int_{spt{\nabla\chi_n}\cap\Omega_2^n}a_n^2\abs {e_{2^n}(z)}^2e^{-u_n(z)}dm(z)\\
&\le a_n^2\bb{\int_{spt{\nabla\chi_n}\cap \beta_n^{+\frac14}}+\int_{spt{\nabla\chi_n}\cap (\ell^{\pm}_n)^{+\frac14}}+\int_{spt{\nabla\chi_n}\cap (\gamma^\pm_n)^{+\frac14}}}\abs {e_{2^n}(z)}^2e^{-u_n(z)}dm(z)\\
&\le a_n^2\bb{m\bb{\beta_n^{+\frac14}}\cdot \exp\bb{-e^{20n}}\underset{z\in \beta_n^{+\frac14}}\max \abs{e_{2^n}(z)}^2+\int_{spt{\nabla\chi_n}\cap (\ell^{\pm}_n)^{+\frac14}}\abs {e_{2^n}(z)}^2dm(z)+\int_{spt{\nabla\chi_n}\cap (\gamma^\pm_n)^{+\frac14}}\abs {e_{2^n}(z)}^2dm(z)},
\end{align*}
as $u_n(z)\geq 0$ for all $z$. To bound the last two integrals, let $z\in (\gamma^\pm_n)^{+\frac14}\cup (\ell^{\pm}_n)^{+\frac14}$. Then there exists $\zeta\in \frac14\D$ so that $z=re^{i\xi_n}+\zeta$, where
$$
\xi_n=	\begin{cases}
			t_n,& z\in (\gamma^+_n)^{+\frac14}, \Im(z)>0,\\
			-t_n,& z\in (\gamma^-_n)^{+\frac14}, \Im(z)<0,\\
			t_n+\Theta_n,& z\in (\ell_n^{+})^{+\frac14},\\
			-\bb{t_n+\Theta_n},& z \in (\ell_n^{-})^{+\frac14}.
		\end{cases}
$$
Since $\Re(z)>0$, $\Arg(z)=\arctan\bb{\frac{\Im(z)}{\Re(z)}}$ and so
\begin{equation}\label{eq_arg}
\Arg(z)=\Arg\bb{re^{i \xi_n}+\zeta}=\xi_n+\Arg\bb{1+\frac{e^{-i\xi_n}\zeta}r}\in \bb{\xi_n-\frac1{r}, \xi_n +\frac1{r}},
\end{equation}
as
$$
\abs{\Arg\bb{1+\frac{e^{-i\xi_n}\zeta}r}}=\abs{\arctan\bb{\frac{\Im\bb{1+\frac{e^{-i\xi_n}\zeta}r}}{\Re\bb{1+\frac{e^{-i\xi_n}\zeta}r}}}}\le \arctan\bb{\frac{\frac{1}{4r}}{1-\frac{1}{4r}}}\le \frac1{r}.
$$
We conclude that if $z\in (\ell^{\pm}_n)^{+\frac14} \cup (\gamma^\pm_n)^{+\frac14}$, then since $t_n=4\pi\cdot 2^{-n}\bb{n+2+\frac16}$ and $r_n=\frac{128n}{t_n}+1$ we get that $\abs{2^n\arg(z)-2^n\xi_n}<\frac{5\pi}{128},$ implying that $\cos(2^n\arg(z))<-\frac14$. Therefore,
$$\abs{e_{2^n}(z)}\le\exp\bb{-\frac{r^{2^n}}4}.$$
Using this estimate, we conclude that
\begin{align*}
	I_2& \le a_n^2\bb{m\bb{\beta_n^{+\frac14}}\cdot \exp\bb{-e^{20n}}\exp\bb{2\bb{\bb{r_n+1}^2+\bb{r_n\cdot \bb{t_n+2\Theta_n}}^2}^{2^{n-1}}}+\frac12\integrate {r_n+\frac14}\infty{ \exp\bb{-\frac{r^{2^n}}4}}r}\\
	&\le \frac{2^{-10n}}{2C^2}
\end{align*}
for all $n\geq N_1$ if $N_1$ is large enough, since we chose $a_n=\exp\bb{-\bb{r_n+3}^{2^n}}$. Overall,  we get that
$$
\int_\C\abs {g_n(z)}^2e^{-u_n(z)}dm(z)\le C^2\bb{I_1+I_2}\le 2^{-10n}
$$
for all $n\geq N_1$ as long as $N_1$ is large enough.
\subsection{Defining the function $f_n$:}
Let $\alpha_n$ be the solution given by H\"ormander's Theorem, Theorem \ref{thm:Hormander}, with $g_n(z)$ and $u_n(z)$ described above.  We define the function
$$
f_n(z):=h_n(z)-\alpha_n(z).
$$
First, note that for every $z\in\C$,
$$
\bar\partial f_n(z)=\bar\partial\bb{h_n(z)-\alpha_n(z)}=\bar\partial h_n(z)-\bar\partial\alpha_n(z)=0
$$
in distribution. In particular, $f$ is entire; see Observation \ref{obs_Hormander}. To conclude the proof of Lemma \ref{lem:sequence_liminf}, we will show properties 2 and 3 hold for $f_n$.
\subsubsection{Property 2 holds:}
Recall that  $\chi_n(w)=1$ whenever $w\in P^{-\frac14}_n\cup S_n^{-\frac14}\cup F_n^{-\frac14}$, and that
$u_n$ is supported in $B_n\cup T_n$. Since $P_n^{-\frac12}\cup S_n^{-\frac12}\cup F_n^{-\frac12} \subset \C\setminus(B_n\cup T_n)$, we have that if $z\in P_n\inv\cup S_n\inv\cup F_n^{-1}$ then $B\bb{z,\frac13}\subset\bset{\chi_n\equiv 1}\cap\bset{u_n\equiv 0}$.
 Both $h_n$ and $f_n$ are holomorphic in $B\bb{z,\frac13}$, and therefore
\begin{align*}
\abs{f_n(z)-h_n(z)}^2&=\frac{81}{\pi^2}\abs{\integrate{B\bb{z,\frac13}}{}{\bb{f_n(w)-h_n(w)}}m(w)}^2=\frac{81}{\pi^2}\abs{\integrate{B\bb{z,\frac13}}{}{\alpha_n(w)}m(w)}^2\\
&\le \frac9\pi\integrate{B\bb{z,\frac13}}{}{\abs{\alpha_n(w)}^2\cdot\frac{e^{-u_n(w)}}{\bb{1+\abs w^2}^2}\cdot \bb{1+\abs w^2}^2 }m(w)\\
&\le \bb{1+\bb{\frac13+\abs z}^2}^2\cdot\frac9\pi\integrate{\C}{}{\abs{\alpha_n(w)}^2\cdot\frac{e^{-u_n(w)}}{\bb{1+\abs w^2}^2}}m(w)\\
&\le \bb{1+\bb{\frac13+\abs z}^2}^2\cdot\frac9{2\pi}\integrate{\C}{}{\abs{g_n(w)}^2\cdot e^{-u_n(w)}}m(w)\le \bb{2^{-4n}}^2\bb{1+\abs z}^4.
\end{align*}
We see that if $z\in P_n\inv\cup S_n\inv\cup F_n^{-1}$, then
$$
\abs{f_n(z)-h_n(z)}\le2^{-4n}\bb{1+\abs z}^2,
$$
concluding the proof of property 2.
\vspace{-10pt}
\subsubsection{Property 3 holds:} Using a similar computation, we see that if $z\in \C\setminus \bb{P_n\inv\cup S_n\inv\cup F_n^{-1}}$, then, using Cauchy-Schwarz,
\begin{align*}
\abs{f_n(z)}^2&=\frac{1}{\pi^2}\abs{\integrate{B(z,1)}{}{f_n(w)}m(w)}^2=
 \frac{1}{\pi^2}\abs{\integrate{B(z,1)}{}{\bb{h_n(w)-\alpha_n(w)}}m(w)}^2\\
 &\leq \frac1\pi\integrate{B(z,1)}{}{\abs{h_n(w)-\alpha_n(w)}^2}m(w)\\
&\le 3\bb{ \underset{w\in B(z,1)}\sup \abs{h_n(w)}^2+\frac1\pi\integrate{B(z,1)}{}{\abs{\alpha_n(w)}^2} m(w)}\\
&\le 3\bb{\underset{w\in B(z,1)}\sup \abs{h_n(w)}^2+\underset{w\in B(z,1)}\sup e^{u_n(w)}\bb{1+\abs w^2}^2\frac1{\pi}\int_\C\abs{\alpha_n(w)}^2\cdot\frac{e^{-u_n(w)}}{\bb{1+\abs w^2}^2}dm(w)}\\
&\le 3\bb{\underset{w\in B(z,1)}\sup \abs{h_n(w)}^2+\exp\bb{\exp\bb{2\pi\bb{\abs z+1}}}\bb{1+\bb{1+\abs z}^2}^2\frac1{2\pi}\cdot 2^{-10n}}.
\end{align*}

To conclude the proof one should bound $\underset{w\in B(z,1)}\sup \abs{h_n(w)}$ for $z\in \C\setminus\bb{P_n\inv\cup S_n\inv\cup F_n^{-1}}\subset F_{n-1}$. We look at four cases: 
\begin{itemize}
\item If $B(z,1)\subset \Omega^n_1$, then $\abs{h_n(w)} =\abs{f_{n-1}(w)}\leq 2\cdot 2^{-4(n-1)}\bb{1+\abs{z}}^2.$ 
\item If $B(z,1)\subset \Omega^n_2$, then every $w\in B(z,1)$ satisfies $w\in (\ell_n^{\pm})^{+2}\cup \beta_n^{+2}$. If $w\in (\ell_n^{\pm})^{+2}$, then arguing as in \eqref{eq_arg}, with  the bound $\frac{3}{r}$ replacing $\frac1r$, we see that $\abs{h_n(w)}\leq\abs{e_{2^n}(w)}\leq 1$. If $w\in \beta_n^{+2}$, then by the choice of $a_n$, $\abs{h_n(w)} =\abs{a_n e_{2^n}(w)}\leq 1$ as well.
\item If $B(z,1)\subset \Omega^n_3$, then $\underset{w\in B(z,1)}\sup \abs{h_n(w)}=0$. 
\item If $B(z,1)$ intersects more than one of the domains, then we bound
$$\underset{w\in B(z,1)}\sup \abs{h_n(w)}\le 2\cdot 2^{-4(n-1)}\bb{1+\abs{z}}^2+1.$$
\end{itemize}	

Overall, we see that
$$
\frac{\abs{f_n(z)}}{3}\le 2\cdot 2^{-4(n-1)}\bb{1+\abs{z}}^2+1+\exp\bb{\exp\bb{2\pi\bb{\abs z+1}}}\bb{1+\bb{1+\abs z}^2}\frac1{2\pi}\cdot 2^{-5n}\le \bb{1+\abs z}^2\exp\bb{\exp\bb{7\abs z}},
$$
for $n\geq N_1$ large enough.

To conclude the proof of property 3 and of Lemma \ref{lem:sequence_liminf}, we shall choose $N_1$ so that the inequality above holds, and so that the bounds on the integrals $I_1$ and $I_2$ from Subsection 3.1 hold.
\hspace*{\fill}\qedsymbol{}\quad

\section{The limiting function- The proof of Theorem \ref{thm:max_min_mod}}
\begin{obs} \label{obs_rho} For every $n\in\N$ for which Observation \ref{obs_N0} holds, the function $J_n\colon [r_{n-1}, \infty)\to\R$ defined by
	\begin{equation*}
 J_n(r) := a_{n-1}e_{2^{n-1}}(r) -a_ne_{2^n}(r)+3\bb{1+r}^2\cdot 2^{-4(n-1)}
	\end{equation*}
	has a zero $\rho_n\in (r_n+3, r_n+4)$ and is strictly decreasing for $r\geq r_{n-1}$.
\end{obs}
\begin{proof}
First, note that $J_n(r_n+3)>0$, since 
$a_ne_{2^n}(r_n+3)=1\ll a_{n-1}e_{2^{n-1}}(r_n+3)$, while $J_n(r_n+4)<0$, as
$$a_{n-1}e_{2^{n-1}}(r_n+4)+3\bb{1+r_n+4}^2\cdot 2^{-4(n-1)} \overset{(\star)}\leq a_ne_{2^n}(r_n+4),$$ 
where $(\star)$ holds if
$$
(r_n+4)^{2^{n-1}}-\bb{r_{n-1}+3}^{2^{n-1}}+1<(r_n+4)^{2^n}-\bb{r_n+3}^{2^n},
$$
which holds if 
$$
(r_n+4)^{-2^{n-1}}+\bb{\frac{r_n+3}{r_n+4}}^{2^n}< (r_n+4)^{-2^{n-1}}+\frac1{e^{\frac1{12}}}<1,
$$
and the latter holds by Observation \ref{obs_N0}(b). By the intermediate-value theorem, there exists  $\rho_n\in\bb{r_n+3,r_n+4}$ so that $J_n(\rho_n)=0$. A calculation shows the second part of the statement.
\end{proof}

The first proposition shows that if $r>\rho_n$, then $M_{f_n}(r)$ is obtained inside $S_n\inv$ or is bounded above by some constant depending on $r$.
\begin{prop}\label{prop_firstbound}
For every $n\ge N_1$ and for every $r\geq \rho_n$,
$$
M_{f_n}(r)\leq \max\left\{\underset{z\in S_n\inv\atop \abs z=r}\max\abs{f_n(z)}, 3\bb{1+r}^2\bb{\exp\bb{\exp\bb{ 7 r}}+\sumit k 1 n 2^{-4k}}\right\}.$$
\end{prop}
\begin{proof}
We will show this by induction on $n$. For $n=N_1$, $M_{f_{N_1}}(r)=a_{N_1}M_{e_{2^{N_1}}}(r)=a_{N_1}\exp(r^{2^{N_1}})$, which is attained at the point $ re^{i\bb{t_{N_1}+2^{-N_1}2\pi \bb{\lfloor N_1/2 \rfloor+2/3}}}\in S^{-1}_{N_1}$, and the claim holds. 

Suppose now that the statement holds for some $n-1 \geq N_1$. Note that since $\C=\bb{P_n\inv\cup S_n\inv\cup F_n^{-1}} \uplus \bb{\C\setminus\bb{P_n\inv\cup S_n\inv\cup F_n^{-1}}}$,
$$
M_{f_n}(r)=\max\bset{\underset{z\in S_n\inv\atop \abs z=r}\max\abs{f_n(z)},\underset{z\in P_n\inv\atop \abs z=r}\max\abs{f_n(z)},\underset{z\in F_n^{-1}\atop \abs z=r}\max\abs{f_n(z)},\underset{z\in \C\setminus\bb{P_n\inv\cup S_n\inv\cup F_n^{-1}}\atop \abs z=r}\max\abs{f_n(z)}}=:\max\bset{M_1,\;M_2,\;M_3,\;M_4}.
$$
We need to show that for every $j\in\bset{2,3,4}$ we have
$$
M_j\le \max\bset{M_1, 3 \bb{1+r}^2\bb{\exp\bb{\exp\bb{7 r}}+\sumit k 1 n 2^{-4k}}}.
$$
First, let us find a lower bound for $M_1$:
\begin{align*}
M_1&=\underset{z\in S_n\inv\atop \abs z=r}\max\abs{f_n(z)}\ge \underset{z\in S_n\inv\atop \abs z=r}\max\abs{h_n(z)}-\underset{z\in S_n\inv\atop \abs z=r}\max\abs{f_n(z)-h_n(z)}\ge a_n\exp\bb{r^{2^n}}-\bb{1+r}^2\cdot 2^{-4n}\\
&= \exp\bb{r^{2^n}-\bb{r_n+3}^{2^n}}-\bb{1+r}^2\cdot 2^{-4n}.
\end{align*}
We will show it dwarfs the rest of the bounds.

Following the induction assumption and property 2 of the sequence $\bset{f_n}$, because $r\geq \rho_n>\rho_{n-1}$,
\begin{align*}
M_2&=\underset{z\in P_n\inv\atop \abs z=r}\max\abs{f_n(z)}\le\underset{z\in P_n\inv\atop \abs z=r}\max\abs{f_{n-1}(z)}+ \bb{1+r}^2\cdot 2^{-4n}\le M_{f_{n-1}}(r)+ \bb{1+r}^2\cdot 2^{-4n}\\
&\le \max\bset{\underset{z\in S_{n-1}\inv\atop \abs z=r}\max\abs{f_{n-1}(z)},3 \bb{1+r}^2\bb{\exp\bb{\exp\bb{ 7 r}}+\sumit k 1 {n-1} 2^{-4k}}}+\bb{1+r}^2\cdot 2^{-4n}.
\end{align*}
Now, following property 2 of the sequence $\bset{f_n}$,
\begin{align*}
\underset{z\in S_{n-1}\inv\atop \abs z=r}\max\abs{f_{n-1}(z)}+\bb{1+r}^2\cdot 2^{-4n}&\le \exp\bb{r^{2^{n-1}}-\bb{r_{n-1}+3}^{2^{n-1}}}+\bb{1+r}^2\bb{2^{-4(n-1)}+2^{-4n}}\\
&\overset{(\star)}\le\exp\bb{r^{2^n}-\bb{r_n+3}^{2^n}}-\bb{1+r}^2\cdot 2^{-4n}\le M_1,
\end{align*}
where $(\star)$ holds by Observation \ref{obs_rho}.
We see that as long as $r \geq \rho_n$,
$$
M_2\le\max\bset{M_1, 3 \bb{1+r}^2\bb{\exp\bb{\exp\bb{7 r}}+\sumit k 1 n 2^{-4k}}}.
$$
Next, using the same logic
$$
M_3=\underset{z\in F_n^{-1}\atop \abs z=r}\max\abs{f_n(z)}\le \bb{1+r}^2\cdot 2^{-4n}<\max\bset{M_1, 3\bb{1+r}^2\bb{\exp\bb{\exp\bb{ 7 r}}+\sumit k 1 n 2^{-4k}}}.
$$
Lastly, following property 3 of the sequence $\bset{f_n}$,
$$
M_4=\underset{z\in \C\setminus\bb{P_n\inv\cup S_n\inv\cup F_n^{-1}}\atop \abs z=r}\max\abs{f_n(z)}\le  3\bb{1+r}^2\exp\bb{\exp\bb{7 r}}\le \max\bset{M_1,3 \bb{1+r}^2\bb{\exp\bb{\exp\bb{7 r}}+\sumit k 1 n 2^{-4k}}},
$$
concluding the proof.
\end{proof}
\begin{cor}\label{cor:max_at_S_n}
For every $n\in\N$ large enough and every $r\in [\rho_n, \rho_{n+1}]$,
$$M_{f_n}(r)=\underset{z\in S_n\inv\atop \abs z=r}\max\abs{f_n(z)}> 3\bb{1+r}^2\bb{\exp\bb{\exp\bb{ 7 r}}+1}.$$
\end{cor}
\begin{proof}
Recall from Observation \ref{obs_N0}(a) and Observation \ref{obs_rho} that $$r\in\sbb{\rho_n,\rho_{n+1}}\subset\bb{r_n+3, r_{n+1}+4}\subset\bb{10\cdot 2^n+3,22\cdot 2^n+4}.$$ Moreover, by  Observation \ref{obs_N0}(d), for any such $r$, $3(1+r)^2\cdot 2^{-4(n-1)}\leq 2^{-n}$. This leads to 
\begin{align*}
	&3\bb{1+r}^2\bb{\exp\bb{\exp\bb{7 r}}+1}\le 3\bb{5+22\cdot 2^n}^2\bb{\exp\bb{\exp\bb{7 \bb{22\cdot 2^n+4}}}+1}\overset{(\star)}\le \exp\bb{\exp\bb{8\cdot22\cdot 2^n}}\\
	&\overset{(\star)}\le \exp\bb{\frac12(10\cdot 2^n+3)^{2^{n-1}}}-1\le \exp\bb{\frac12\cdot\rho_n^{2^{n-1}}}-1 \le 
\exp\bb{\rho_n^{2^{n-1}}-(r_{n-1}+3)^{2^{n-1}}}-1\\
	&= a_{n-1}e_{2^{n-1}}(\rho_n)-1\leq  a_ne_{2^n}(\rho_n)-2^{-n}-1 \leq 
	\underset{z\in S_{n}\inv\atop\abs z=r}\max\; a_{n}e_{2^{n}}(z)-1< \underset{z\in S_n\inv\atop \abs z=r}\max\abs{f_n(z)},
\end{align*}
where $(\star)$ holds for all $n$ large enough.  This combined with Proposition \ref{prop_firstbound} shows that the statement holds as long as $n$ is sufficiently large.
\end{proof}

The next observation shows that the sequence $\bset{f_n}$, constructed in the previous subsection converges locally uniformly to an entire function.
\begin{obs}
The sequence $\bset{f_n}$ constructed in Lemma \ref{lem:sequence_liminf} converges locally uniformly to an entire function,~$f$.
\end{obs}
\begin{proof}
We will show that $\bset{f_n}$ forms a Cauchy sequence. Fix $r>0$ and $\eps>0$ and let $n_0$ be so that for every $n\ge n_0$ we have $r_n\ge r+1$ and $2^{-n}<\eps$. Then, for every $n>n_0$ we have $B(0,r)\subset P_n\inv$ and following property 2 of the sequence, for every $m>n>n_0$ we have
\begin{align*}
\underset{z\in B(0,r)}\sup\abs{f_n(z)-f_m(z)}&\le \sumit k n{m-1} \underset{z\in B(0,r)}\sup\abs{f_k(z)-f_{k+1}(z)}\le \sumit k n{m-1} \underset{z\in B(0,r)}\sup 2^{-4k}\bb{1+\abs z}^2\\
&=(1+r)^2 \sumit k n{m-1}2^{-4k}\le r_n^2\cdot \frac{16}{15}\cdot 2^{-4n}\le 150\cdot 2^{2n}\cdot 2^{-4n}<2^{-n}<\eps
\end{align*}
by the way $n_0$ was chosen. We conclude that the sequence $\bset{f_n}$ is a Cauchy sequence and converges locally uniformly to an entire function.
\end{proof}
\begin{prop}\label{prop_Mf}For every $n\in\N$ large enough, for every $r\in [\rho_n, \rho_{n+1}]$,  
		\begin{equation*}
			\abs{M_{f}(r)-a_ne_{2^n}(r)}\le 2^{-n+2}.
		\end{equation*}	
	\end{prop}
	\begin{proof} Let $N$ be large enough so that Lemma \ref{lem:sequence_liminf}, Corollary \ref{cor:max_at_S_n}, and Observation~\ref{obs_N0} hold for all $n\ge N$. Fix $n\ge N$ and let $r$ be as in the statement. Then, by Lemma \ref{lem:sequence_liminf}, Corollary \ref{cor:max_at_S_n}, and Observation~\ref{obs_N0}(d),
\begin{equation}\label{eq_maxPn}
		\left\vert \underset{z\in P^{-1}_{n+1}\atop \abs z=r}\max\abs{f(z)}-  \underset{z\in P_{n+1}\inv\atop \abs z=r}\max\abs{f_n(z)}\right \vert = \left\vert \underset{z\in P^{-1}_{n+1}\atop \abs z=r}\max\abs{f(z)}-  \underset{z\in S_{n}\inv\atop \abs z=r}\max\abs{f_n(z)}\right \vert\leq (1+r)^2\sum_{j\geq n}\underset{z\in P^{-1}_{n+1}\atop \abs z=r}\max \vert f_{j+1}(z)-f_j(z)\vert \leq 2^{-n},
	\end{equation}
since $S^{-1}_{n}\subset P^{-1}_{n+1}$.

Note that for $r\in [\rho_n, r_{n+1}-1]$, by definition, $\{z \colon \vert z \vert =r\}\subset  P^{-1}_{n+1}$. On the other hand, if $r\in [r_{n+1}-1, \rho_{n+1}]$, $\{z \colon \vert z \vert =r\}$ might also intersect $S^{-1}_{n+1}$, $F^{-1}_{n+1}$, or the complement of these sets. Recall from Lemma \ref{lem:sequence_liminf} that
		$$
		\vert f_{n+1}(z)\vert \leq 	
		\begin{cases}
			(1+\vert z\vert)^2 \cdot 2^{-4(n+1)},& z\in F^{-1}_{n+1},\\ 
			3 \bb{1+\abs z}^2\exp\bb{\exp\bb{7 \abs z}}, & z\in\C\setminus\bb{P_{n+1}\inv\cup S_{n+1}\inv\cup F_{n+1}^{-1}}, \\
		\end{cases}
		$$
		and so, for $r\in [r_{n+1}-1, \rho_{n+1}]$, since $\{z \colon \vert z \vert =r\}\subset  P^{-1}_{n+2}$, using Lemma \ref{lem:sequence_liminf}, Corollary \ref{cor:max_at_S_n} and \eqref{eq_maxPn},
		\begin{align*}
			\underset{z\notin (P^{-1}_{n+1}\cup S^{-1}_{n+1})\atop \abs z=r}\max\abs{f(z)} &\leq  \underset{z\notin (P^{-1}_{n+1}\cup S^{-1}_{n+1})\atop \abs z=r}\max\abs{f_{n+1}(z)}+\sum_{j\geq n+1}\underset{z\in P^{-1}_{n+2}\atop \abs z=r}\max \vert f_{j+1}(z)-f_j(z)\vert \\
			&\leq \underset{z\in S_n\inv\atop \abs z=r}\max\abs{f_n(z)}-1+(1+r)^2\sum_{j\geq n+1}2^{-4j}\\
			&\leq
			\underset{z\in P_{n+1}\inv\atop \abs z=r}\max\abs{f_n(z)} -2^{-n}.
		\end{align*}
		Hence, we have  
		\begin{equation*}\label{eq_maxnotSn}
			\left\vert \underset{z\notin S^{-1}_{n+1}\atop \abs z=r}\max\abs{f(z)}-  \underset{z\in S_{n}\inv\atop \abs z=r}\max\abs{f_n(z)}\right \vert \leq 2^{-n}.
 		\end{equation*}
			Now, $ \underset{\abs z=r}\max \vert e_{2^n}(z)\vert=e_{2^n}(r)$ is attained at the point $re^{i\bb{t_n+2^{-n}2\pi \bb{\lfloor n/2 \rfloor+2/3}}}\in S^{-1}_{n}$, and therefore
		\begin{equation*}
			\left\vert \underset{z\notin S^{-1}_{n+1}\atop \abs z=r}\max\abs{f(z)}-a_n e_{2^n}(r)\right \vert \leq 	\left\vert \underset{z\notin S^{-1}_{n+1}\atop \abs z=r}\max\abs{f(z)}-  \underset{z\in S_{n}\inv\atop \abs z=r}\max\abs{f_n(z)}\right \vert+ 	\left\vert \underset{z\in S^{-1}_{n}\atop \abs z=r}\max\abs{f_n(z)} -a_n e_{2^n}(r) \right\vert \leq 2^{-n}+(1+r)^2\cdot 2^{-4n}<2^{-n+1}.
		\end{equation*}
		On the other hand, arguing similarly, since $S^{-1}_{n+1}\subset P^{-1}_{n+2}$,
	\begin{align*}
	\left\vert \underset{z\in S^{-1}_{n+1}\atop \abs z=r}\max\abs{f(z)}-a_{n+1} e_{2^{n+1}}(r)\right \vert &\leq 	\left\vert \underset{z\in S^{-1}_{n+1}\atop \abs z=r}\max\abs{f(z)}-  \underset{z\in S_{n+1}\inv\atop \abs z=r}\max\abs{f_{n+1}(z)}\right \vert+ 	\left\vert \underset{z\in S^{-1}_{n+1}\atop \abs z=r}\max\abs{f_{n+1}(z)} -a_{n+1}e_{2^{n+1}}(r) \right\vert \\
	&\leq (1+r)^2\sum_{j\geq n+1}2^{-4j}+(1+r)^2\cdot 2^{-4n+1}<  2^{-n+1}.
\end{align*}
It follows from these two inequalities that for any $r\in [\rho_n, \rho_{n+1}]$, 
\begin{equation*}
M_f(r)> \max\left\{ a_n e_{2^n}(r)-2^{-n+1}, a_{n+1}e_{2^{n+1}}(r)-2^{-n+1} \right\}>a_n e_{2^n}(r)-2^{-n+1},
\end{equation*}
while,  using Observation \ref{obs_rho} and Observation \ref{obs_N0}(d),
$$M_f(r) \leq \max\left\{ a_n e_{2^n}(r)+2^{-n+1}, a_{n+1}e_{2^{n+1}}(r)+2^{-n+1} \right\}\leq a_n e_{2^n}(r)+2\cdot 2^{-n+1},$$
and the statement follows readily.
\end{proof}
We can finally show that the function $f$ satisfies the requirements of Theorem \ref{thm:max_min_mod}.
\begin{lem}\label{lem_final}
The function $f$ satisfies the requirements of Theorem \ref{thm:max_min_mod}.
\end{lem}

\begin{proof}
	Fix $\eps>0$, and let $N$ be chosen at the end of the proof. We may assume without loss of generality that $N$ is large enough such that for all $n\ge N$ and $r\in[\rho_n, \rho_{n+1}]\subset\bb{10\cdot 2^n+3,22\cdot 2^n+3}$, 
	\begin{equation} \label{eq_n_geq_eps}
		\exp\bb{-\frac{r^{2^n}}2}+ \frac{2 r^2}{2^{4n}}+\frac{1}{2^{n-2}}<\exp\bb{-\frac{(10\cdot 2^n+3)^{2^n}}2}+ \frac{2(22\cdot 2^{n}+3)^2}{2^{4n}}+\frac{1}{2^{n-2}}<\eps.
	\end{equation}
	In order to prove the lemma, it suffices to show that for every $n>N$ and $r>\rho_n$ there are at least $n-2$ disjoint arcs in $\{\vert z\vert=r\}$ whose boundary points satisfy $\abs{f(w)}<\eps$ and each of them contains at least one point where $\abs{f(z)}\geq M_f(r)-\eps$.
	
	For every $n\geq N$, $r\in\sbb{\rho_n,\rho_{n+1}}$, and $0\le k\le n-1$ we define the arcs
	$$
	I_k^n(r):=\bset{re^{it}, t\in\bb{t_n+2\pi k\cdot 2^{-n}, t_n+2\pi (k+1)\cdot 2^{-n}}}.
	$$
	Note that as long as $k\in\bset{1,\ldots,n-1}$ we see that
	$$
	t_n+\Theta_n=t_n+2\pi\cdot2^{-n}\cdot\bb{n+\frac{1}6}\ge t_n+2\pi n\cdot 2^{-n}\ge  t_n+2\pi (k+1)\cdot 2^{-n}.
	$$
	In particular, if $N$ is chosen large enough, for every $n\geq N$ and $r\geq \rho_n$,  $I_1^n(r),\cdots,I_{n-2}^n(r)\subset S^{-1}_n$. 
	
	Since the intervals $I^n_j(r)$ are contained in $S_{n}^{-1}\subset P_{\nu}^{-1}$ for $\nu > n$, 
	$$
	(1+r)^2 \cdot 2^{-4\nu}\geq 	\begin{cases}
		\abs{f_n\bb{re^{i\bb{t_n+2\pi k\cdot 2^{-n}}}}-h_n\bb{re^{i\bb{t_n+2\pi k\cdot 2^{-n}}}}},& \nu=n,\\
		\abs{f_{\nu}\bb{re^{i\bb{t_n+2\pi k\cdot 2^{-n}}}}-f_{\nu-1}\bb{re^{i\bb{t_n+2\pi k\cdot 2^{-n}}}}},& \text{otherwise}.
	\end{cases}
	$$
	
	We conclude that if $r\in\sbb{\rho_n,\rho_{n+1}}$, then at the end-points of these arcs,
	\begin{align*}
		\abs{f\bb{re^{i\bb{t_n+2\pi k\cdot 2^{-n}}}}}&\le \abs{f_n\bb{re^{i\bb{t_n+2\pi k\cdot 2^{-n}}}}}+\sumit \nu n\infty\abs{f_\nu\bb{re^{i\bb{t_n+2\pi k\cdot 2^{-n}}}}-f_{\nu+1}\bb{re^{i\bb{t_n+2\pi k\cdot 2^{-n}}}}}\\
		&\le  \abs{h_n\bb{re^{i\bb{t_n+2\pi k\cdot 2^{-n}}}}}+\bb{1+r}^2\sumit \nu n\infty 2^{-4\nu}\le \exp\bb{-\frac{r^{2^n}}2}+2r^2\cdot 2^{-4n}<\eps,
	\end{align*}
	since, for every $n$ for which Observation \ref{obs_N0} holds,
	$$
	\abs{e_{2^n}\bb{re^{i\bb{t_n+2\pi k\cdot 2^{-n}}}}}=\exp\bb{r^{2^n}\cos\bb{2^n\bb{t_n+2\pi k\cdot 2^{-n}}}}=\exp\bb{r^{2^n}\cos\bb{2^n\cdot t_n}}=\exp\bb{-\frac{r^{2^n}}2}.
	$$
	On the other hand, for every $1\leq k\leq n-2$, $re^{i\bb{t_n+2^{-n}2\pi \bb{k+2/3}}}\in I_k^n(r)$ while
	\begin{align*}
		\abs{f\bb{re^{i\bb{t_n+2^{-n}2\pi \bb{k+2/3}}}}}&\ge \abs{f_n\bb{re^{i\bb{t_n+2^{-n}2\pi \bb{k+2/3}}}}}-\sumit \nu n\infty\abs{f_\nu\bb{re^{i\bb{t_n+2^{-n}2\pi \bb{k+2/3}}}}-f_{\nu+1}\bb{re^{i\bb{t_n+2^{-n}2\pi \bb{k+2/3}}}}}\\
		&\ge  \abs{h_n\bb{re^{i\bb{t_n+2^{-n}2\pi \bb{k+2/3}}}}}-\bb{1+r}^2\sumit \nu n\infty 2^{-4\nu}=a_n\exp\bb{r^{2^n}}-2r^2\cdot 2^{-4n}\\
		&\ge M_f(r)-\eps,
	\end{align*}
	following \eqref{eq_n_geq_eps} and Proposition \ref{prop_Mf}  as long as $n$ is large enough. To conclude the proof, it is enough to choose $N$ so that  \eqref{eq_n_geq_eps}, Observation \ref{obs_N0}, and Proposition \ref{prop_Mf} hold.
\end{proof}


\bigskip\bigskip\bigskip\bigskip\bigskip\bigskip

\noindent A.G.: Department of Mathematics, Northwestern University, Evanston, IL 60208, United States
\newline{\tt https://orcid.org/0000-0002-6957-9431}
\newline{\tt adiglucksam@gmail.com}

\bigskip\bigskip
\noindent L.P.S.: Department of Mathematics, The University of Manchester, Manchester, M13 9PL, United Kingdom 
\newline {\tt https://orcid.org/0000-0003-4039-5556}
\newline{\tt leticia.pardosimon@manchester.ac.uk} 

\end{document}